\documentclass{amsart}
\usepackage{amssymb}
\usepackage{graphicx}
\usepackage[shortlabels]{enumitem}
\usepackage{multirow}
\usepackage{amsmath,color}
\usepackage{hyperref}
\usepackage{url}
\usepackage{longtable}
\usepackage{tikz}

\theoremstyle{plain}
\newtheorem{theorem}{Theorem}[section]
\newtheorem{prop}[theorem]{Proposition}

\newtheorem{lemma}[theorem]{Lemma}
\newtheorem{corollary}[theorem]{Corollary}

\theoremstyle{definition}
\newtheorem{definition}{Definition}

\theoremstyle{remark}

{

\newcommand{\vl}{\; | \;}

\begin{document}
\title[$3$-eigenvalue graphs with second largest eigenvalue at most $1$]{Graphs with three eigenvalues and \\ second largest eigenvalue at most $1$}

\author[X.-M. Cheng]{ Xi-Ming Cheng}
\author[G. R. W. Greaves]{ Gary R. W. Greaves$^{\dag}$ }
\thanks{$^\dag$G.R.W.G. was supported by JSPS KAKENHI; 
grant number: 26$\cdot$03903}
\author[J. H. Koolen]{ Jack H. Koolen$^\ddag$ }
\thanks{$^\ddag$J.H.K. is partially supported by the `100 talents' program of the Chinese Academy of Sciences,
and by the National Natural Science Foundation of China (No. 11471009).}


\address{School of Mathematical Sciences, 
University of Science and Technology of China,
Hefei, Anhui, 230026, P.R. China}
\email{xmcheng@mail.ustc.edu.cn}

\address{Research Center for Pure and Applied Mathematics,
Graduate School of Information Sciences, 
Tohoku University, Sendai 980-8579, Japan}
\email{grwgrvs@gmail.com}

\address{Wen-Tsun Wu Key Laboratory of CAS, School of Mathematical Sciences, 
University of Science and Technology of China,
Hefei, Anhui, 230026, P.R. China}
\email{koolen@ustc.edu.cn}

\subjclass[2010]{05E30, 05C50}

\keywords{three distinct eigenvalues, strongly regular graphs, nonregular graphs, second largest eigenvalue}

\begin{abstract}
	We classify the connected graphs with precisely three distinct eigenvalues and second largest eigenvalue at most $1$.
\end{abstract}

\maketitle

\section{Introduction}

Let $\Gamma$ be a connected graph with adjacency matrix $A$.
The eigenvalues of $\Gamma$ are defined as the eigenvalues of $A$.
Suppose that $\Gamma$ has $r$ distinct eigenvalues then, since they are real, we can arrange them as
$\theta_0 > \theta_1 > \dots > \theta_{r-1}$.
By the famous Perron-Frobenius theorem, the multiplicity of the largest eigenvalue $\theta_0$ is always one.
Therefore $\theta_1$ is the second largest eigenvalue of $\Gamma$.
In this paper we classify all connected graphs with precisely three distinct eigenvalues and second largest eigenvalue at most $1$.

This work then lies at the interface of two branches of research; the study of graphs with second largest eigenvalue at most $1$ and the study of graphs with three distinct eigenvalues.
The problem of characterising graphs with second largest eigenvalue at most $1$ was posed by A. J. Hoffman~\cite{Cvet82} and since the early 80s there has been sporadic progress on the problem.
We refer the reader to \cite{Li11} and the references therein for background of the problem. 
In 1995 W. Haemers~\cite{MK} posed the problem of studying graphs with three eigenvalues and this problem has also received some intermittent attention \cite{CGGK,Dam3ev,MK}.

For both of these problems regular graphs provide a convenient starting point.
Indeed, a regular graph with second largest eigenvalue at most $1$ is the complement of a regular graph with smallest eigenvalue at least $-2$.
Graphs with smallest eigenvalue at least $-2$ have been extensively studied since the beautiful classification theorem of P. J. Cameron et al.~\cite{CGSS76}.
On the other hand, a regular graph with three distinct eigenvalues is well-known~\cite[Lemma 10.2.1]{God01} to be a strongly regular graph and such graphs have also received a great deal of attention.
Hence it is easy to see that regular graphs with three distinct eigenvalues and second largest eigenvalue at most $1$ are strongly regular graphs whose complements have smallest eigenvalue at least $-2$.

Among nonregular graphs with three distinct eigenvalues and second largest eigenvalue at most $1$, some examples immediately occur to us.
First the complete bipartite graphs $K_{a,b}$ (with $a > b \geqslant 1$) which have eigenvalues $\sqrt{ab} > 0 > -\sqrt{ab}$.
We also have two sporadic examples which we call the \textbf{Petersen cone} and the \textbf{Fano graph} (see Figure~\ref{fig:peterfan}).
For details of their construction see \cite{Dam3ev}.

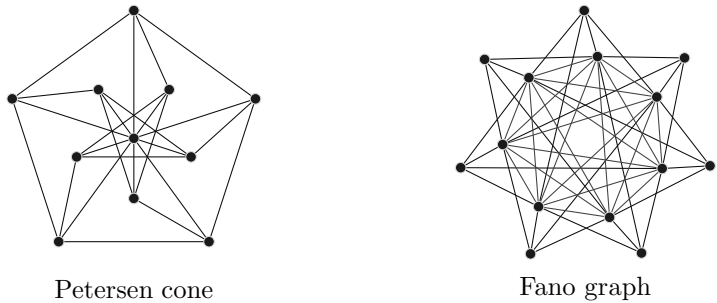
\begin{figure}[htbp]
	\centering
		\begin{tikzpicture}
			\tikzstyle{vertex}=[circle,thin,draw=black!15,fill=black!90, inner sep=0pt, minimum width=4pt]
			\tikzstyle{edge}=[draw=black!95,-]
			\tikzstyle{fedge}=[draw=black!75,-]
			\begin{scope}
				\newdimen\rad
				\rad=1.7cm
				\newdimen\radi
				\radi=1.1cm
				\def\angle{51.4285}
				\def\offsetI{39}
				\def\offsetII{30}
				\foreach \x in {0,...,6}
				{
					\draw[edge] (\offsetI+\angle*\x:\rad) -- (\offsetII+\angle*\x:\radi);
					\draw[edge] (\offsetI+\angle*\x:\rad) -- (\offsetII+\angle*\x+3*\angle:\radi);
					\draw[edge] (\offsetI+\angle*\x:\rad) -- (\offsetII+\angle*\x + \angle:\radi);
					\draw[edge] (\offsetI+\angle*\x:\rad) -- (\offsetII+\angle*\x - \angle:\radi);
			    }
				\foreach \x in {0,...,6}
				{
					\foreach \y in {\x,...,6}
					{
						\draw[fedge] (\offsetII+\angle*\x:\radi) -- (\offsetII+\angle*\y:\radi);
				    }
			    }
				\foreach \x in {0,...,6}
				{
					\draw (\offsetI + \angle*\x:\rad) node[vertex] {};
					\draw (\offsetII + \angle*\x:\radi) node[vertex] {};
			    }
				\node at (0,-2) {Fano graph};
			\end{scope}
	
			\begin{scope}[xshift=-6cm]
				\newdimen\rad
				\rad=1.7cm
				\newdimen\radi
				\radi=0.8cm
				\def\angle{72}
				\def\offsetI{18}
				\def\offsetII{-18}
				\foreach \x in {0,...,4}
				{
					\draw[edge] (\offsetI+\angle*\x:\rad) -- (\offsetII+\angle*\x:\radi);
					\draw[edge] (\offsetII+\angle*\x:\radi) -- (\offsetII+\angle*\x + 2*\angle:\radi);
					\draw[edge] (\offsetI+\angle*\x:\rad) -- (\offsetI+\angle*\x + \angle:\rad);
					\draw[edge] (\offsetI+\angle*\x:\rad) -- (0:0);
					\draw[edge] (\offsetII+\angle*\x:\radi) -- (0:0);
			    }
				\foreach \x in {0,...,4}
				{
					\draw (\offsetI + \angle*\x:\rad) node[vertex] {};
					\draw (\offsetII + \angle*\x:\radi) node[vertex] {};
			    }
				\draw (0:0) node[vertex] {};
				\node at (0,-2) {Petersen cone};
			\end{scope}
		\end{tikzpicture}
	\caption{The Petersen cone and the Fano graph.}
	\label{fig:peterfan}
\end{figure}

We will see that the complete bipartite graphs, the Petersen cone, and the Fano graph are in fact the only nonregular graphs with three distinct eigenvalues and second largest eigenvalue at most $1$.
This is our main contribution.

\begin{theorem}\label{thm:main}
	Let $\Gamma$ be a connected nonregular graph with three distinct eigenvalues and second largest eigenvalue at most $1$.
	Then $\Gamma$ is one of the following graphs.
	\begin{enumerate}[(a)]
		\item A complete bipartite graph;
		\item The Petersen cone;
		\item The Fano graph.
	\end{enumerate}
\end{theorem}

We remark that this theorem is dual to the classification of nonregular graphs with precisely three distinct eigenvalues and smallest eigenvalue at least $-2$ due to E. R. van Dam~\cite{Dam3ev}.
A consequence of our main theorem is a partial answer to a question D. de Caen \cite[Problem 9]{Dam05} (also see \cite{Dam14}) who asked if graphs with three distinct eigenvalues have at most three distinct valencies.
In fact it follows from our theorem that if a graph with three distinct eigenvalues has more than \emph{two} distinct valencies then it must have second largest eigenvalue greater than one.

We call a graph \textbf{biregular} if it has precisely two distinct valencies.
Recently, the following specialisation of Theorem~\ref{thm:main} to biregular graphs was established.

\begin{prop}[\mbox{\cite[Proposition 3.11]{CGGK}}]\label{pro:secondev1}
	Let $\Gamma$ be a connected biregular graph with three distinct eigenvalues and second largest eigenvalue at most $1$.
	\begin{enumerate}[(a)]
		\item A complete bipartite graph;
		\item The Petersen cone;
		\item The Fano graph.
	\end{enumerate}
\end{prop}

Observe that our main theorem (Theorem~\ref{thm:main}) is a relaxation of the hypothesis of Proposition~\ref{pro:secondev1}.
Using well-known results, we can also incorporate the regular case.
Indeed, J. J. Seidel found the following classification.

\begin{theorem}[\mbox{\cite[Theorem 14]{sei}}]
	\label{thm:Seidel}
	Let $\Gamma$ be a connected strongly regular graph with smallest eigenvalue at least $-2$.
	Then $\Gamma$ is either
		 a triangular graph $T(m)$ for $m \geqslant 5$;
		 an $(m \times m)$-grid for $m \geqslant 3$;
		 the Petersen graph;
		 the Shrikhande graph;
		 the Clebsch graph;
		 the Schl\"{a}fli graph;
		 or one of the three Chang graphs.
\end{theorem}

As we observed above, regular graphs with three distinct eigenvalues and second largest eigenvalue at most $1$ are strongly regular graphs, whose complements have smallest eigenvalue at least $-2$.
From Theorem~\ref{thm:Seidel} we know all coconnected strongly regular graphs whose complements have smallest eigenvalue at least $-2$.
The non-coconnected strongly regular graphs are complete multipartite graphs, which have second largest eigenvalue $0$.
Consequently we have the following.

\begin{corollary}\label{cor:main}
	Let $\Gamma$ be a connected graph with three distinct eigenvalues and second largest eigenvalue at most $1$.
	Then $\Gamma$ is one of the following graphs.
	\begin{enumerate}[(a)]
		\item A complete bipartite graph;
		\item The Petersen cone;
		\item The Fano graph;
		\item A complete multipartite regular graph;
		\item The complement of one of the graphs in Theorem~\ref{thm:Seidel}.
	\end{enumerate}
\end{corollary}

The paper is organised as follows.
In Section~\ref{sec:prelim} we develop some basic theory for graphs with three eigenvalues and state some preliminary results.
In Section~\ref{sec:cones} we prove Theorem~\ref{thm:main} for cones (see Section~\ref{sec:cones} for the definition of a cone).
We reduce the proof of Theorem~\ref{thm:main} to a finite search in Section~\ref{sec:bounding_the_smallest_eigenvalue}, and in Section~\ref{sec:description_of_our_computation} we describe the algorithm that we used to perform the finite search.

\section{Graphs with three distinct eigenvalues}
\label{sec:prelim}

In this section we develop some basic properties of graphs that have three distinct eigenvalues.
For fixed $\theta_0 > \theta_1 > \theta_2$, define the set $\mathcal G(\theta_0,\theta_1,\theta_2)$ of connected graphs having precisely three distinct eigenvalues $\theta_0$, $\theta_1$, and $\theta_2$.
Let $\Gamma$ be a graph in $\mathcal G(\theta_0,\theta_1,\theta_2)$ and let $A$ denote its adjacency matrix.
By Perron-Frobenius theory, $\theta_0$ is a simple eigenvalue whose eigenvectors have all entries of the same sign.
We denote by $\alpha$ the eigenvector for $\theta_0$ satisfying the equation
\begin{equation}
	\label{eqn:alpha}
	(A - \theta_1 I)(A-\theta_2 I) = \alpha \alpha^\top.
\end{equation}
Let $x$ and $y$ be vertices of $\Gamma$.
We denote by $\nu_{x,y}$ the number of common neighbours of $x$ and $y$.
By Equation~\eqref{eqn:alpha}, we can see that the degree $d_x$ of $x$ is given by $d_x = \alpha_x^2 - \theta_1 \theta_2$ and that $\nu_{x,y}$ is given by
\begin{equation}
	\label{eqn:commonneighbs}
	\nu_{x,y} = (\theta_1 + \theta_2)A_{x,y} + \alpha_x\alpha_y.
\end{equation}
Moreover, multiplying Equation~\eqref{eqn:alpha} by $A$ reveals the following equation.
\begin{equation}
	\label{eqn:triangles}
	A^3 = (\theta_1^2+\theta_2^2 + \theta_1\theta_2)A -(\theta_1+\theta_2)\theta_1\theta_2 I + (\theta_0+\theta_1+\theta_2)\alpha \alpha^\top.
\end{equation}

If $\Gamma$ is a graph with three distinct eigenvalues and second largest eigenvalue at most $1$ then the spectrum of $\Gamma$ is very restricted.
Indeed, below we will first show that, excepting complete bipartite graphs, the eigenvalues must be integers.

Van Dam deduced spectral properties of graphs with three eigenvalues that do not have an integral spectrum.

\begin{lemma}[\mbox{\cite[Proposition 3]{Dam3ev}}]\label{lem:vanDamNonInt}
	Let $G$ be a non-bipartite graph in $\mathcal G(\theta_0,\theta_1,\theta_2)$ with one of $\theta_0$, $\theta_1$, or $\theta_2$ not integral.
	Then $\theta_1, \theta_2 = (-1 \pm \sqrt{b})/2$ for some $b \equiv 1 \mod 4$.
\end{lemma}

Using this result we can show that graph in $\mathcal G(\theta_0,\theta_1,\theta_2)$ with an irrational second largest eigenvalue must have smallest eigenvalue greater than $-2$.

\begin{corollary}\label{cor:between01}
	Let $G$ be a graph in $\mathcal G(\theta_0,\theta_1,\theta_2)$ with $\theta_1 \in (0,1)$.
	Then $\theta_2 > -2$.
\end{corollary}
Note that we do not need the non-bipartite condition because we assume that $\theta_1 > 0$.
Indeed, since their spectrum is symmetric about zero, if a graph in $\mathcal G(\theta_0,\theta_1,\theta_2)$ is bipartite then $\theta_1 = 0$.
\begin{proof}
	By Lemma~\ref{lem:vanDamNonInt}, since $\theta_1$ is nonintegral, $\theta_1,\theta_2 = (-1\pm\sqrt{b})/2$ for some $b \equiv 1 \mod 4$.
	Moreover, since $\theta_1 \in (0,1)$, we have that $b = 5$.
	Therefore we have $\theta_2 = (-1-\sqrt{5})/2 > -2$.
\end{proof}

Graphs with smallest eigenvalue at least $-2$ are well understood.
In particular, for graphs with three eigenvalues, we have the following result. 

\begin{theorem}[\cite{Dam3ev},\cite{sei}]\label{thm:class-2}
	Let $G$ be a graph in $\mathcal G(\theta_0,\theta_1,\theta_2)$ with $\theta_2 \geqslant -2$.
	Then $G$ is one of the following graphs:
	\begin{enumerate}
		\item a graph from Theorem~\ref{thm:Seidel};
		\item a graph from Theorem 7 in \cite{Dam3ev}.
	\end{enumerate}
\end{theorem}

By checking the spectra of the graphs Theorem~7 in \cite{Dam3ev}, we obtain the following corollary.

\begin{corollary}\label{cor:inteigs}
	Let $\Gamma \in \mathcal G(\theta_0,\theta_1,\theta_2)$ be nonbipartite and nonregular with $\theta_1 \leqslant 1$.
	Then $\theta_1 = 1$ and both $\theta_0$ and $\theta_2$ are integers.
\end{corollary}

Conversely, it also follows from Lemma~\ref{lem:vanDamNonInt} that if there is a graph in $\mathcal G(\theta_0,1,\theta_2)$ then $\theta_0$ and $\theta_2$ must be integers.
Using Equation~\eqref{eqn:commonneighbs} we can obtain the next corollary.


\begin{corollary}\label{cor:intparams}
	Let $\Gamma$ be a graph in $\mathcal G(\theta_0,1,\theta_2)$.
	Then there exists some $\omega \in \mathbb{N}$ such that each vertex $v$ has $\alpha_v = \beta_v \sqrt{\omega}$ for some $\beta_v \in \mathbb{N}$.
\end{corollary}

It is well-known~\cite{smith70} that complete multipartite graphs are characterised by the property that their second largest eigenvalue is at most $0$.
In particular, $\mathcal G(\theta_0,0,\theta_2)$ consists exclusively of complete multipartite graphs.
Furthermore, every regular graph in $\mathcal G(\theta_0,\theta_1,\theta_2)$ is strongly regular.
Therefore, with a view towards classification, we need only consider nonregular graphs in $\mathcal G(\theta_0,1,\theta_2)$.
Hence we define the set $\mathcal H(s,t)$ to consist of the nonregular graphs in $\mathcal G(s,1,-t)$.
In view of Corollary~\ref{cor:intparams}, for a graph $\Gamma \in \mathcal H(s,t)$ and a vertex $v \in V(\Gamma)$, we write $\omega(\Gamma)$ to denote the squarefree part of $\alpha_v^2$.
Moreover, we write $\delta(\Gamma)$ and $\Delta(\Gamma)$ respectively to denote the smallest and largest valency of $\Gamma$.

Van Dam and Kooij~\cite{DK} showed that the number $n$ of vertices of a connected graph $\Gamma$ with diameter 2 with spectral radius $\rho$ satisfies $n \leqslant \rho^2 +1$ with equality if and only if $\Gamma$ is a Moore graph of diameter 2 or $\Gamma$ is $K_{1, n-1}$. 
Since Moore graphs are regular we have the following lemma.

\begin{lemma}\label{lem:rhobound}
	Let $\Gamma$ be an $n$-vertex graph in $\mathcal H(s,t)$.
	Then $\delta(\Gamma) < s < \Delta(\Gamma)$ and $n \leqslant s^2 + 1$ with equality if and only if $\Gamma$ is $K_{1,n-1}$.
\end{lemma}

Bell and Rowlinson exhibited an upper bound for the number of vertices in terms of the multiplicity of one of its eigenvalues.

\begin{theorem}[\mbox{\cite[Theorem 2.3]{br03}}]\label{thm:bell} 
	Let $\Gamma$ be an $n$-vertex graph with an eigenvalue $\theta$ with multiplicity $n-l$ for some positive integer $l$.
	Then either $\theta \in \{0, -1\}$ or $n \leqslant \frac{l(l+1)}{2}$.
\end{theorem}

Let $\Gamma$ be a graph in $\mathcal H(s,t)$ and suppose $\Gamma$ has $r$ distinct valencies $k_1 < \dots < k_r$.  
We define $V_i := \{ v \in V(\Gamma) \; | \; d_v = k_i \}$ and we say that each $k_i$ has multiplicity $n_i := |V_i|$ for $i \in \{1,\dots,r\}$.
We will refer to the $n_i$ as the \textbf{valency multiplicities} (see Section~\ref{sec:description_of_our_computation}).
The subsets $V_i$ form what we call the \textbf{valency partition} of $\Gamma$.

Let $\pi = \{\pi_1,\dots,\pi_r\}$ be a partition of the vertices of $\Gamma$.
For each vertex $x$ in $\pi_i$, write $d_{ij}^{(x)}$ for the number of neighbours of $x$ in $\pi_j$.
Then we write $b_{ij} = 1/|\pi_i|\sum_{x \in \pi_i} d_{ij}^{(x)}$ for the average number of neighbours in $\pi_j$ of vertices in $\pi_i$.
The matrix $B_\pi := (b_{ij})$ is called the \textbf{quotient matrix} of $\pi$ and $\pi$ is called \textbf{equitable} if for all $i$ and $j$, we have $d_{ij}^{(x)} = b_{ij}$ for each $x \in \pi_i$.
We will use repeatedly properties of the quotient matrices of partitions of the vertex set of a graph and we refer the reader to Godsil and Royles' book~\cite[Chapter 9]{God01} for the necessary background on equitable partitions and interlacing.

In the case when $\Gamma$ has at most three valencies we can use the following result.

\begin{lemma}[See \cite{Dam3ev}]\label{lem:equitable}
	Let $\Gamma$ be a graph in $\mathcal H(s,t)$ that has at most three distinct valencies.
	Then the valency partition of $\Gamma$ is equitable.
\end{lemma}

In the remainder of the paper our graphs $\Gamma$ will always be connected.
Moreover, we will reserve the letter $A$ to denote the adjacency matrix of $\Gamma$, $n$ to denote the number of vertices of $\Gamma$, $\alpha$ to denote the Perron-Frobenius eigenvector as defined by Equation~\eqref{eqn:alpha}, and $m$ to denote the multiplicity of the smallest eigenvalue of $\Gamma$.
In particular, for a graph $\Gamma \in \mathcal G(\theta_0,\theta_1,\theta_2)$, since the trace of $A$ is $0$ we have
\begin{equation}
	\label{eqn:multeq}
	m = (\theta_0 - (n-1)\theta_1)/(\theta_1-\theta_2).
\end{equation}

\section{Cones} 
\label{sec:cones}

A graph $\Gamma$ (with $n$ vertices) is called a \textbf{cone} if it has a vertex $v$ with degree $d_v = n-1$.
In this section we consider the case when $\Gamma \in \mathcal H(s,t)$ is a cone.

By Proposition~\ref{pro:secondev1}, we need only consider cones with at least three distinct valencies.
In fact, by the following result of Van Dam, we need only consider cones with \emph{precisely} three valencies.

\begin{theorem}[See \cite{Dam3ev}, \mbox{\cite[Corollary 2.4]{CGGK}}]\label{thm:cone3}
	Let $\Gamma$ be a cone in $\mathcal H(s,t)$.
	Then $\Gamma$ has at most three distinct valencies.
\end{theorem}

Now we give a technical result about cones with three valencies.

\begin{lemma}\label{lem:techcone}
	Let $\Gamma$ be cone in $\mathcal G(\theta_0,\theta_1,\theta_2)$ with three valencies.
	Set $A:=-\theta_1\theta_2$ and $B:=-\theta_1-\theta_2-1$.
	Then
	\begin{enumerate}[(a)]
		\item $5A+4B+1 < n \leqslant (A+B+1)^2 + A+1$;
		\item if $B > 0$ then $n \leqslant (A+B)^2/B +3(A+B)+1$;
		\item $\theta_0^2-(n-2-A-B)\theta_0 - (n-1)(1-A-B) = 0$.
		\item $n \geqslant 3(A+B)+\sqrt{8(A+B)(A+B-1)}$.
	\end{enumerate}
\end{lemma}
\begin{proof}
    Let $k_0 = n-1$, $k_1$, and $k_2$ with $k_1 > k_2$ be the three valencies of $\Gamma$ and define $V_i = \{ v \in V(\Gamma) \; | \; d_v = k_i \}$ and $n_i = |V_i|$.

	Let $x$ be a vertex with $d_x = n-1$ and let $v$ be a vertex in $V_1 \cup V_2$.
	Then $\nu_{v,x} = d_v - 1$ and hence we have $(\alpha_x-\alpha_v)\alpha_v = -(\theta_1+1)(\theta_2+1)$.
	Therefore there are two possible values for $\alpha_v$, which sum to $\alpha_x$.
	Note that, since $\theta_2 \ne -1$, the entry $\alpha_v$ cannot be equal to $\alpha_x$.
	Thus we must have $n_0 = 1$.
	
	Let $y \in V_1$ and $z \in V_2$ such that $y \not \sim z$.
	(Such vertices exist since the complement of $\Gamma$ has two connected components \cite[Theorem 2.2: Claim 3]{CGGK}.)
	Then we have 
	\begin{align}
		\alpha_y+\alpha_z &= \alpha_x = \sqrt{n-1+\theta_1\theta_2}; \label{eqn:al1} \\
		\alpha_y\alpha_z &= -(\theta_1+1)(\theta_2+1) = A+B. \label{eqn:al2}
	\end{align}
	Since $(\alpha_y+\alpha_z)^2 > 4\alpha_y\alpha_z$, using Equations~\eqref{eqn:al1} and \eqref{eqn:al2}, we have $n-1+\theta_1\theta_2> 4(A+B)$ from which we obtain our first bound $n > 5A +4B + 1$.
	On the other hand we have
	\[
		 n-1+\theta_1\theta_2 = (\alpha_y+\alpha_z)^2 = 2\alpha_y\alpha_z + \alpha_y^2 + \alpha_z^2 = 2(A+B) + \alpha_y^2 + (A+B)^2/\alpha_y^2.
	\]
	Since $\alpha_z \geqslant 1$ and $\alpha_y > \alpha_z$, the inequality $\sqrt{A+B} \leqslant \alpha_y \leqslant A+B$ also follows from Equation~\eqref{eqn:al2}.
	Therefore $n \leqslant (A+B+1)^2+A+1$ and we have shown Item (a).
		
	Since $y \not \sim z$, we have $\nu_{y,z} = \alpha_y\alpha_z \leqslant d_z = \alpha_z^2 - \theta_1\theta_2$.
	Hence $\alpha_z^2 \geqslant B$.
	Furthermore, if $B > 0$ then, using Equation~\eqref{eqn:al2}, we have $\alpha_y = (A+B)/\alpha_z \leqslant (A+B)/\sqrt{B}$.
	This implies Item (b).
	
	Since $\alpha$ is an eigenvector for $\theta_0$, we can write $n_1\alpha_y + n_2\alpha_z = \theta_0\alpha_x$.
	Combine this with the equation $n = 1 + n_1 + n_2$ to obtain
    \[
	    \begin{aligned}
	         n_1 = \frac{\theta_0\alpha_x-(n-1)\alpha_z}{\alpha_y-\alpha_z}; 
	    \end{aligned} \quad
	    \begin{aligned}
	         n_2 = \frac{(n-1)\alpha_y - \theta_0\alpha_x}{\alpha_y-\alpha_z}.
	    \end{aligned}
    \]
	From the trace of the square of the adjacency matrix of $\Gamma$, we have
	\[
		k_0 +n_1 k_1 +n_2 k_2 =\theta_0^{2} + (n-1-m)\theta_1^2 +m \theta_2^{2}.
	\]
	Using Equation~\eqref{eqn:multeq}, the expressions for $n_i$ and writing the $k_i$ in terms of $n$, $\alpha_y$, $\alpha_z$, and $\theta_1\theta_2$, we can deduce Item (c).
	
	Finally, using Item (c), and the fact that $\theta_0$ is real, we deduce the discriminant $(n-2-A-B)^2-4(n-1)(A+B-1)$ is nonnegative.
	Therefore, $(n-3A-3B)^2 \geqslant 8(A+B)(A+B-1)$, and Item (d) follows using the lower bound of  Item (a).
\end{proof}

Now we can classify the cones in $\mathcal H(s,t)$.

\begin{theorem}\label{thm:petersencone}
     Let $\Gamma$ be a cone in $\mathcal H(s,t)$.
	 Then $\Gamma$ is the Petersen cone.
 \end{theorem}
 \begin{proof}
     By Theorem~\ref{thm:cone3}, we know $\Gamma$ has at most three valencies.
	 If $\Gamma$ has only two distinct valencies then the theorem follows from Proposition~\ref{pro:secondev1}.
     Assume $\Gamma$ has three distinct valencies.
	 
	 Item (a) of Lemma~\ref{lem:techcone} gives $n <10t-5+4/(t-2)$
	 On the other hand, Item (d) of Lemma~\ref{lem:techcone} gives $n \geqslant 6t-6+ \sqrt{8(4t^2 -10t +6)}$.
	 Comparing bounds we see that $t \leqslant 6$.
	 It remains to check case by case when $t = 3, 4, 5$, and $6$.
	
	Suppose $t=3$.
	Let $x$ and $y$ be vertices with $d_y< d_x < n-1$.
	Then $\alpha_x\alpha_y = 2(t-1) = 4$ and hence, by Corollary~\ref{cor:intparams}, we have $(\alpha_x,\alpha_y) \in \{(4,1), (2\sqrt{2},\sqrt{2})\}$.
	Using the proof of Lemma~\ref{lem:techcone}, we also know that $\alpha_x+\alpha_y = \sqrt{n-1-t}$ and $s^2-(n-2t)s - (n-1)(2t-3) = 0$.
	It follows that $(\alpha_x,\alpha_y) = (2\sqrt{2},\sqrt{2})$, $n = 22$, and $s = 7$.
	Let $n_1$ denote the number vertices $v$ with $d_v = d_x$.
	Then, using the formula for $n_1$ in the proof of Lemma~\ref{lem:techcone}, we have $n_1 = \frac{s\sqrt{n-1-t}-(n-1)\alpha_y}{\alpha_x-\alpha_y} = 0$.
	Hence, $\Gamma$ only has two distinct valencies and this is a contradiction.
	The cases for $t = 4, 5$, and $6$ follow similarly.
 \end{proof}

Henceforth we may assume our graphs are neither cones nor biregular graphs.
We denote by $\mathcal H^\prime(s,t)$ the set $\mathcal H(s,t)$ without cones and biregular graphs.

In the sequel we will take advantage of the following structural characterisation of cones with three eigenvalues.

\begin{theorem}[\mbox{\cite[Theorem 2.2]{CGGK}}]\label{thm:disconn}
	Let $\Gamma$ be a graph in $\mathcal H(s,t)$.
	Then $\Gamma$ is a cone if and only if the complement of $\Gamma$ is disconnected.
\end{theorem}


\section{Reduction to a finite search} 
\label{sec:bounding_the_smallest_eigenvalue}

In this section we reduce the proof of Theorem~\ref{thm:main} to a finite search.
The key results of this section are Lemma~\ref{lem:boundVerts} and Theorem~\ref{thm:4vals}.
We begin by stating a result from the authors' previous work on graphs with three eigenvalues~\cite{CGGK}.
This result essentially follows from Equation~\eqref{eqn:commonneighbs}.

\begin{lemma}[\mbox{\cite[Lemma 2.6]{CGGK}}]\label{lem:alphabound}
	Let $\Gamma$ be a graph in $\mathcal H(s,t)$ and let $x$ and $y$ be vertices having valencies $d_x > d_y$.
	Then we have the following inequalities:
	\begin{enumerate}
		\item if $x \sim y$ then $\alpha_x-1 \leqslant (\alpha_x-\alpha_y)\alpha_y \leqslant 2(t-1)$ and $\alpha_x\alpha_y \geqslant t-1$;
		\item if $x \not \sim y$ then $\alpha_x-1 \leqslant (\alpha_x-\alpha_y)\alpha_y \leqslant t$.
	\end{enumerate}
\end{lemma}

Lemma~\ref{lem:alphabound} enables us to bound the size of the valencies of graphs in $\mathcal H^\prime(s,t)$.

\begin{corollary}\label{cor:maxDeg}
	Let $\Gamma$ be a graph in $\mathcal H^\prime(s,t)$.
	Then for all vertices $x \in V(\Gamma)$ we have $\alpha_x \leqslant t+1$.
\end{corollary}
\begin{proof}
	Suppose every vertex with valency $d$ was adjacent to every vertex of valency not equal to $d$.
	Then the complement of $\Gamma$ would be disconnected and by Theorem~\ref{thm:disconn}, $\Gamma$ would be a cone.
	Since $\Gamma$ is not a cone it must have two nonadjacent vertices $x \not \sim y$ with $d_x > d_y$.
\end{proof}

	The next corollary follows easily from Lemma~\ref{lem:alphabound} by checking the minima of negative parabolas with bounded domains.

\begin{corollary}\label{cor:tboundal}
	Let $\Gamma$ be a graph in $\mathcal H^\prime(s,t)$ and let $x$ and $y$ be vertices with $d_x = \Delta(\Gamma)$ and $d_x > d_y$.
	Then
	\begin{enumerate}[(a)]
		\item If $x \sim y$ and $\alpha_x \geqslant \alpha_y+4 \geqslant 8$ then $\alpha_x \leqslant (t+7)/2$. 
		\item If $x \not \sim y$ and $\alpha_x \geqslant \alpha_y + 2 \geqslant 4$ then $\alpha_x \leqslant (t+4)/2$. 
	\end{enumerate}
\end{corollary}

Let $\Gamma$ be a graph in $\mathcal H^\prime(s,t)$ and let $v$ be a vertex of $\Gamma$.
Recall that $\omega(\Gamma)$ denotes the squarefree part of $\alpha_v^2$.

\begin{corollary}\label{cor:psige3}
	Let $\Gamma$ be a graph in $\mathcal H^\prime(s,t)$ and let $x$ be a vertex with $d_x = \Delta(\Gamma)$ where $\omega(\Gamma) \geqslant 3$.
	Then $\alpha_x \leqslant (t+3)/\sqrt{3}$.
\end{corollary}
\begin{proof}
	By Theorem~\ref{thm:disconn}, there must be a vertex $y$ not adjacent to $x$.
	Furthermore, since $\omega(\Gamma) \geqslant 3$, we have $\sqrt{3} \leqslant \alpha_y \leqslant \alpha_x - \sqrt{3}$.
	Using Lemma~\ref{lem:alphabound}, we have $\sqrt{3}(\alpha_x-\sqrt{3}) \leqslant (\alpha_x-\alpha_y)\alpha_y \leqslant t$, and the lemma follows.
\end{proof}

The next result is a structural result about graphs with second largest eigenvalue at most $1$.
Note that this result does not depend on the fact that $\Gamma$ has three eigenvalues.

\begin{lemma}\label{lem:cells}
	Let $\Gamma$ be a graph in $\mathcal H(s,t)$ with $x \sim y$ adjacent vertices.
	Let $\pi$ be a vertex partition with cells $C_1 = \{x,y\}$, $C_2 = \{z \in V(\Gamma)\backslash C_1 \vl z \sim x \text{ or } z \sim y \}$, and $C_3 = \{z \in V(\Gamma) \vl z \not \sim x \text{ and } z \not \sim y \}$.
	Then the induced subgraph on $C_3$ has maximum degree~one.
\end{lemma}
\begin{proof}
	Suppose there is a vertex $z \in C_3$ having degree $d \geqslant 1$ in the induced subgraph on $C_3$.
	Define the vertex partition $\pi^\prime$ to consist of the cells $C_1^\prime = \{x,y\}$, $C_2^\prime = \{z\} \cup \{ w \in C_3 \vl w \sim z \}$, and $C_3^\prime = V(\Gamma)\backslash \{ C_1^\prime \cup C_2^\prime \}$.
	Let $\Gamma^\prime$ be the subgraph of $\Gamma$ induced on $C_2^\prime$.
	Set $\mathfrak d = \sum_{v \in V(\Gamma^\prime)} d^\prime_v/(d+1)$, where $d^\prime_v$ denotes the degree of the vertex $v$ in $\Gamma^\prime$.
	Note that $\mathfrak d \geqslant 1$ and equality implies that $d^\prime_v = 1$ for all $v \in V(\Gamma^\prime)$.
	
	Now $\pi^\prime$ has quotient matrix
	\[
		Q = \begin{pmatrix}
			1 & 0 & q_1 \\
			0 & \mathfrak d & q_2 \\
			q_5 & q_4 & q_3
		\end{pmatrix}.
	\]
	By interlacing the second largest eigenvalue $\theta$ of $Q$ is at least $1$.
	Moreover, since the eigenvalues of $Q$ interlace with those of $\Gamma$, we see that $\theta$ is bounded above by $1$.
	Hence $\theta = 1$, and furthermore $\mathfrak d=1$.
\end{proof}

Let $\Gamma$ be a graph in $\mathcal H(s,t)$.
Recall that $m$ denotes the multiplicity of the eigenvalue $-t$.

\begin{lemma}\label{lem:indep}
	Let $\Gamma$ be a graph in $\mathcal H(s,t)$.
	Then the independence number of $\Gamma$ is at most $m$.
\end{lemma}
\begin{proof}
	Let $C$ be an independent set in $\Gamma$.
	The eigenvalues of the subgraph induced on $C$ are all zero and they interlace with the eigenvalues of $\Gamma$.
	Hence we have $|C| \leqslant m$.
\end{proof}

\begin{lemma}\label{lem:boundVerts}
	Let $\Gamma$ be a graph in $\mathcal H(s,t)$ and let $x \sim y$ be adjacent vertices.
	Then $d_x + d_y - \nu_{x,y} \geqslant n - 2 m$.
	Moreover, if $d_x = \Delta(\Gamma)$ then
	\begin{align}
		n &\leqslant \frac{(t+1)(\alpha_x^2 - \alpha_x + 3t) + 2\alpha_x^2+2t-2}{t-1}; \label{eqn:nuppal} \\
		(\sqrt{2n}- (t+1))^2 &> (t-1)(t-2) - 2\alpha_x^2. \label{eqn:nlowal}
	\end{align}
\end{lemma}

\begin{proof}
	Let $\pi$ be a vertex partition with cells $C_1 = \{x,y\}$, $C_2 = \{z \in V(\Gamma)\backslash C_1 \vl z \sim x \text{ or } z \sim y \}$, and $C_3 = \{z \in V(\Gamma) \vl z \not \sim x \text{ and } z \not \sim y \}$.
	By Lemma~\ref{lem:cells}, the subgraph induced on $C_3$ consists of isolated vertices and disjoint edges.
	Hence, by Lemma~\ref{lem:boundVerts}, $|C_3| \leqslant 2 m$.
	On the other hand $|C_3| = n - (d_x + d_y - \nu_{x,y})$.
	Thus, $n - 2m \leqslant d_x + d_y - \nu_{x,y} = \alpha_x^2+3t-1-\alpha_y(\alpha_x - \alpha_y)$.
	By Lemma~\ref{lem:alphabound} we have $n - 2m \leqslant \alpha_x^2+3t-1-(\alpha_x - 1)$.
	Furthermore, if $d_x = \Delta(\Gamma)$ then, using Equation~\eqref{eqn:multeq} and Lemma~\ref{lem:rhobound}, we can write $m = (n-1+s)/(t+1) \leqslant (n-1+d_x)/(t+1) = (n-1+\alpha_x^2+t)/(t+1)$.
	The first two inequalities follow.
	
	Now we show the third inequality.
	By Theorem~\ref{thm:bell} we have $m > \sqrt{2n} - 3/2$.
	Since the sum of the eigenvalues of $\Gamma$ equal $0$ we have $m(t+1) = n-1+s$.
	Hence $(t+1)(\sqrt{2n} - 3/2) < n-1+s$ and again, if $d_x = \Delta(\Gamma)$ then, by Lemma~\ref{lem:rhobound}, we have $n-1+s \leqslant n-1+\alpha_x^2 + t$.
	Therefore $(t+1)(\sqrt{2n} - 3/2) < n-1+\alpha_x^2 + t$.
\end{proof}

We remark that if $(t-1)(t-2) - 2\alpha_x^2$ is nonnegative then we have the following lower bound for $n$.
\[
	n > \frac{\left ( t + 1 + \sqrt{(t-1)(t-2) -2\alpha_x^2} \right )^2}{2}.
\]
Using our general upper bound for $\alpha_x$, note the following upper and lower bounds for $n$ in terms of $t$.

\begin{corollary}\label{cor:nbound}
	Let $\Gamma$ be a graph in $\mathcal H^\prime(s,t)$.
	Then 
	$$\frac{1}{2}(t-\frac{1}{2})^2 < n < t^2 + 8t + 18 + 18/(t-1).$$
\end{corollary}

\begin{proof}
	Let $x$ be a vertex with $d_x = \Delta(\Gamma)$.
	Then by Corollary~\ref{cor:maxDeg}, we have $\alpha_x \leqslant t+1$.
	Using this with Lemma~\ref{lem:boundVerts} gives the upper bound.
	
	For the lower bound we follow the proof of Lemma~\ref{lem:boundVerts} to obtain
	$(t+1)(\sqrt{2n} - 3/2) < n-1+d_x$.
	Since $\Gamma$ is not a cone, $d_x \leqslant n-2$.
	Whence (since we can assume $n \geqslant 2$), we have $n > \frac{1}{2}(t-\frac{1}{2})^2$.
\end{proof}

By comparing the upper and lower in Lemma~\ref{lem:boundVerts}, one can obtain the following bounds for $t$.

\begin{corollary}\label{cor:tbound}
	Let $\Gamma$ be a graph in $\mathcal H^\prime(s,t)$ and let $x$ be a vertex with $d_x = \Delta(\Gamma)$.
	\begin{enumerate}[(a)]
		\item If $\alpha_x \leqslant (t+4)/2$ then $t \leqslant 15$.
		\item If $\alpha_x \leqslant (t+18)/3$ then $t \leqslant 19$.
		\item If $\alpha_x \leqslant (t+3)/\sqrt{3}$ then $t \leqslant 21$.
		\item If $\alpha_x \leqslant (t+7)/2$ then $t \leqslant 22$.
		\item If $\alpha_x \leqslant (t+10)/2$ then $t \leqslant 29$.
	\end{enumerate}
\end{corollary}

Next we establish a lower bound for the minimum degree of a graph in $\mathcal H^\prime(s,t)$ in terms of $t$.
By Corollary~\ref{cor:intparams}, we know that the minimum degree is at least $t+1$.
In the next two results we improve this lower bound when $t\geqslant 7$ and $t\geqslant 11$.

\begin{theorem}\label{thm:kmin}
	Let $\Gamma$ be a graph in $\mathcal H^\prime(s,t)$ with $t \geqslant 7$.
	Then $\delta(\Gamma) \geqslant t+2$.
\end{theorem}

\begin{proof}
	Suppose that $\delta(\Gamma) = t+1$ and that $x$ is a vertex of $\Gamma$ that has valency $t+1$.
	Since $\Gamma$ is not a cone, we can apply Lemma~\ref{lem:alphabound} to find that the maximum degree $\Delta(\Gamma) \leqslant t^2+3t+1$.
	In terms of the Perron-Frobenius eigenvector $\alpha$, for any vertex $v$, we have $1 \leqslant \alpha_v \leqslant t+1$.
	Moreover, since the pairwise products of $\alpha_i$s are integers and $\alpha_x = 1$, we have that $\alpha_v$ is an integer for all vertices $v$.
	
	Let $w$ be a neighbour of $x$.
	The number of common neighbours $\nu_{w,x} = 1-t+\alpha_w\alpha_x$ is at least zero and hence $\alpha_w \geqslant t-1$.
	Therefore, for all $w \in N_\Gamma(x)$, we have that $\alpha_w \in \{t-1, t, t+1 \}$.
	
	\paragraph{\textbf{Claim 1}} There is no vertex $w \in N_\Gamma(x)$ with $\alpha_w=t+1$. 
	\label{par:claim1}
	
	For a contradiction, assume that $y$ is a neighbour of $x$ with $\alpha_y=t+1$.
	If a vertex $w \in N_\Gamma(x)$ had $\alpha_w = t-1$ then $\nu_{x,w} = 0$.
	But according to Lemma~\ref{lem:alphabound}, $w$ is adjacent to $y$, which is a contradiction.
	Hence, for all $w \in N_\Gamma(x)$, we have $\alpha_w \in \{t, t+1 \}$.
	
	Next suppose that for all $w \in N_\Gamma(x)$ we have $\alpha_w = t+1$.
	Since $\Gamma$ is not biregular, there must be a vertex $z$ (say) in $\Gamma$ such that $2 \leqslant \alpha_z \leqslant t$.
	Now, $z \not \sim x$ and $\nu_{x,z} = \alpha_z \leqslant t$.
	Therefore, there must be some vertex $w \in N_\Gamma(x)$ that is not adjacent to $z$.
	By Lemma~\ref{lem:alphabound}, we have $(\alpha_w-\alpha_z)\alpha_z \leqslant t$, and hence $\alpha_z=t$.
	Thus, $z$ is adjacent to every neighbour of $x$ except for $w$.
	On the other hand, $\nu_{w,z} = t^2 + t = d_z$, which means that $w$ is also adjacent to every other vertex in $N_\Gamma(x)$.
	This gives a contradiction, since $\nu_{w,x} = 2$.
	Therefore at least one neighbour $v$ (say) of $x$ has $\alpha_v = t$.

	The number of common neighbours of $x$ and $v$ is $\nu_{v,x} = 1$.
	Let $z$ be the common neighbour of $x$ and $v$.
	\paragraph{\textbf{Case 1.1}} $\alpha_z = t$. 
	\label{par:_textbf_case_1}
	
	In this case $v \not \sim y$ and so $\nu_{v,y} = t^2 + t = d_v$, which means that $y$ is adjacent to $z$.
	But this is impossible since $\nu_{x,z} = 1$.
	
	
	\paragraph{\textbf{Case 1.2}} $\alpha_z = t+1$. 
	\label{par:_textbf_case_2}
	
	Note that the only common neighbour of $v$ and $x$ is $z$ and the only common neighbours of $x$ and $z$ are $v$ and $y$.
	Let $w$ be another neighbour of $x$.
	We cannot have $\alpha_w = t+1$ since in this case $\nu_{v,w} = t^2 + t = d_v$ which means $w$ would be adjacent to $z$.
	Therefore the remaining neighbours $w$ of $x$ have $\alpha_w = t$, and hence each of these neighbours $w$ has precisely one common neighbour $w^\prime$ (say) with $x$.
	Either $w^\prime$ is $y$ or $w^\prime$ is adjacent to $y$.
	Since $t \geqslant 7$, the vertex $x$ has at least $8$ neighbours including $v$, $y$, and $z$.
	Therefore $x$ has at least $5$ other neighbours $w$ with $\alpha_w = t$, but the number of common neighbours of $x$ and $y$ is $2$.
	This gives a contradiction and completes the proof of Claim 1. 
	
	
	\paragraph{\textbf{Claim 2}} There is no vertex $w \in N_\Gamma(x)$ with $\alpha_w=t$. 
	\label{par:claim2}
	
	By Claim~1, for all $w \in N_\Gamma(x)$, we have that $\alpha_w \in \{t-1, t \}$.
	Suppose for a contradiction, there is a vertex $y \in N_\Gamma(x)$ with $\alpha_y = t$.
	Let $\pi$ be the partition of the vertices of $\Gamma$ with cells $C_1 = \{x,y\}$, $C_2 = \{ w \in V(\Gamma)\backslash C_1 \vl w \sim x \text{ or } w \sim y \}$, and $C_3 = \{ w \in V(\Gamma) \vl w \not \sim x \text{ and } w \not \sim y \}$.
	Let $w \in C_3$ then $\alpha_w \leqslant t$ otherwise $w$ would be adjacent to $y$.
	Moreover, we have $d_w = \alpha_w^2+t \geqslant \nu_{w,x} + \nu_{w,y} - \nu_{x,y}$ from which it follows that $\alpha_w = 1$ or $t$.
	Let $z \in C_2\backslash N_\Gamma(x)$.
	Similarly we have $d_z = \alpha_z^2 + t \geqslant \nu_{x,z} + \nu_{y,z} - \nu_{x,y} = \alpha_z(t+1) - t$. 
	Hence, since $t \geqslant 7$, we have $\alpha_z \in \{1, 2, t-1, t, t+1 \}$.
	
	\paragraph{\textbf{Case 2.1}} $\alpha_z = 2$. 
	\label{par:_textbf_case_11}
	
	Since $t \geqslant 7$, by Lemma~\ref{lem:alphabound}, $z$ is adjacent to every vertex $v$ with $\alpha_v = t-1$ or $t$.
	In particular, $z$ is adjacent to every neighbour of $x$ and hence $z$ has $t+1$ common neighbours with $x$.
	But $\nu_{x,z} = 2$, which is a contradiction.
	
	\paragraph{\textbf{Case 2.2}} $\alpha_z = t+1$. 
	\label{par:_textbf_case_22}
	
	The number of neighbours of $z$ in $C_3$ is $d_z - \nu_{x,z} - \nu_{y,z} + \nu_{x,y} = 2t$.
	Let $a$ and $b$ denote the number of vertices $v \in C_3$ having $\alpha_v = 1$ and $t$ respectively.
	Set $S = C_2 \cap N_\Gamma(x)$.
	By counting the common neighbours with $x$ we see that the number of edges between $S$ and $C_3$ is equal to $a + tb$.
	For each vertex $w \in S$, the number of neighbours in $C_3$ is at least $d_w - \nu_{w,x} - \nu_{w,y}$.
	The vertex $y$ has precisely one neighbour in $S$, hence we find that $a+tb \leqslant 2t-2 + (t-1)(t-1) = t^2 - 1$.
	Therefore $b \leqslant t-1$ and consequently $z$ must have a neighbour $v \in C_3$ with $\alpha_v = 1$.
	But this violates Claim~1.
	
	Thus we know that all vertices $v$ of $\Gamma$ have $\alpha_v \in \{1,t-1,t\}$.
	By Lemma~\ref{lem:equitable}, the valency vertex-partition of $\Gamma$ is equitable.
	Let $Q$ be the quotient matrix given as
	\[
		Q = \begin{pmatrix}
			0 & k_{12} & k_{13} \\
			k_{21} & k_{22} & k_{23} \\
			k_{31} & k_{32} & k_{33}
		\end{pmatrix}.
	\]  
	Note that since the number of common neighbours between any two vertices is nonnegative, no pair of vertices $v$ and $w$ can be adjacent if $\alpha_v = \alpha_w = 1$.
	Since the vector $(1, t-1, t)^\perp$ is an eigenvector of $Q$ for the eigenvalue $s$ we have 
	\begin{align*}
		k_{12}(t-1) + k_{13}t &= s; \\
		k_{12} + k_{13} &= 1 + t.
	\end{align*}
	Hence $s = t^2 - 1 + k_{13}$.
	Furthermore we have
	\begin{align}
		k_{21} + k_{22}(t-1) + k_{23}t &= s(t-1); \label{eqn:secl1} \\
		k_{21} + k_{22} + k_{23} &= t^2 - t + 1. \label{eqn:secl2}
	\end{align}
	Therefore, $s \leqslant t(t^2-t+1)/(t-1) < t^2 + 2$.
	Since $k_{13} \geqslant 2$, we must have $k_{13} = 2$ and $s = t^2+1$.
	By Equations \eqref{eqn:secl1} and \eqref{eqn:secl2}, we find that $k_{21}(t-1) + k_{22} = 1$.
	Hence $k_{21} = 0$, which is impossible.
	
	By Claim~1 and Claim~2, for all $w \in N_\Gamma(x)$, we have that $\alpha_w = t-1$.
	Note that the number of triangles containing $x$ is $0$ and hence, by Equation~\eqref{eqn:triangles}, we have $s = t^2 - 1$.
	Let $y \in N_\Gamma(x)$ and, as before, let $\pi$ be the partition of the vertices of $\Gamma$ with cells $C_1 = \{x,y\}$, $C_2 = \{ w \in V(\Gamma)\backslash C_1 \vl w \sim x \text{ or } w \sim y \}$, and $C_3 = \{ w \in V(\Gamma) \vl w \not \sim x \text{ and } w \not \sim y \}$.
	Let $w \in C_3$ then $\alpha_w \leqslant t-1$ otherwise $w$ would be have more than one neighbour in $C_3$ which is impossible by Lemma~\ref{lem:cells}.
	We have $d_w = \alpha_w^2+t \geqslant \nu_{w,x} + \nu_{w,y} - \nu_{x,y}$ from which it follows that $\alpha_w = 1$ or $t-1$.
	Moreover, by checking common neighbours with $x$ and $y$, we find that the subgraph induced on $C_3$ is $1$-regular.
	Similar to above, let $a$ and $b$ denote the number of vertices $v \in C_3$ having $\alpha_v = 1$ and $t-1$ respectively.
	Since no two vertices $v$ and $w$ with $\alpha_v=\alpha_w=1$ can be adjacent we must have $b \geqslant a$.
	In the same way as above, we count the edges between $C_3$ and $C_2 \cap N_\Gamma(x)$ to find that $a+(t-1)b = t^2$.
	Hence either $a = 1$ and $b = t+1$ or $a=b=t$, which correspond to $C_3$ consisting $t+2$ or $2t$ vertices respectively.
	On the other hand, since the maximum degree of $\Gamma$ is greater than $s$, there must exist a vertex $z \in C_2$ with $\alpha_z \geqslant t$.
	Thus the number of neighbours of $z$ in $C_3$ is $d_z - \nu_{x,z} - \nu_{y,z} \geqslant 2t - 1$.
	This means we must have $a=b=t$ and $z$ is adjacent to at least one vertex $v$ with $\alpha_v=1$.
	But this contradicts Claim~1 or Claim~2 applied to $v$.
\end{proof}

Using techniques similar to the ones used in the previous proof, we improve the lower bound on the minimum valency for $t \geqslant 11$.

\begin{theorem}\label{thm:kmin2}
	Let $\Gamma$ be a graph in $\mathcal H^\prime(s,t)$ with $t \geqslant 11$.
	Then $\delta(\Gamma) \geqslant t+3$.
\end{theorem}

\begin{proof}
	By Theorem~\ref{thm:kmin}, the minimum degree is $\delta(\Gamma) \geqslant t+2$.
	Suppose for a contradiction that $\delta(\Gamma) = t+2$ and that $x$ is a vertex of $\Gamma$ that has valency $t+2$.
	Since $\Gamma$ is not a cone, we can apply Lemma~\ref{lem:alphabound} to find that the maximum degree $\Delta(\Gamma) \leqslant (t^2+6t+4)/2$.
	In terms of the Perron-Frobenius eigenvector, for any vertex $v$, we have $\sqrt{2} \leqslant \alpha_v \leqslant (t+2)/\sqrt{2}$.
	Moreover, since the pairwise products of $\alpha_i$s are integers and $\alpha_x = \sqrt{2}$, we have that $\alpha_v = \sqrt{2}\beta_v$ where $\beta_v$ is an integer for all vertices $v$.
	
	Now we split the proof into two distinct cases depending on the parity of $t$.
	Since these two cases are similar, we will leave the case when $t$ is even to the reader.
	Henceforth we assume that $t$ is odd.
	In this case the bound on $\alpha_v = \sqrt{2}\beta_v$ becomes $1 \leqslant \beta_v \leqslant (t+1)/2$.
	
	Let $w$ be a neighbour of $x$.
	The number of common neighbours $\nu_{w,x} = 1-t+\alpha_w\alpha_x = 1 - t + 2\beta_w$ is at least zero and hence $\beta_w \geqslant (t-1)/2$.
	Therefore, for $w \in N_\Gamma(x)$, we have that $\beta_w \in \{(t-1)/2, (t+1)/2 \}$.
	
	Let $y$ be a neighbour of $x$ and let $\pi$ be a partition of the vertices of $\Gamma$ with cells $C_1 = \{x,y\}$, $C_2 = \{ w \in V(\Gamma)\backslash C_1 \vl w \sim x \text{ or } w \sim y \}$, and $C_3 = \{ w \in V(\Gamma) \vl w \not \sim x \text{ and } w \not \sim y \}$.
	
	\paragraph{\textbf{Claim 1}} There is no vertex $w \in N_\Gamma(x)$ with $\beta_w=(t+1)/2$.
	\label{par:claim21}
	
	For a contradiction, assume that $\beta_y=(t+1)/2$.
	Let $z \in C_2 \backslash N_\Gamma(x)$.
	Note that the number of neighbours of $z$ that are also neighbours of $x$ or $y$ is at least $\nu_{x,z} + \nu_{y,z} - \nu_{x,y}$.
	Therefore $d_z = 2\beta_z^2 + t \geqslant \nu_{x,z} + \nu_{y,z} - \nu_{x,y}$ from which we must have $\beta_z \in \{1,2,(t-1)/2,(t+1)/2 \}$.
	Similarly, for all vertices $w \in C_3$ we have $d_w = 2\beta_w^2 + t \geqslant \nu_{w,x} + \nu_{w,y} - \nu_{x,y}$, and hence we have $\beta_w \in \{1,(t+1)/2 \}$.
	
	Now we show that we cannot have $\beta_z = 2$.
	Suppose for a contradiction that $\beta_z = 2$.
	Since $t \geqslant 11$, by Lemma~\ref{lem:alphabound}, $z$ is adjacent to every vertex $v$ with $\beta_v = (t-1)/2$ or $(t+1)/2$.
	In particular, $z$ is adjacent to every neighbour of $x$ and hence $z$ has $t+2$ common neighbours with $x$.
	But $\nu_{x,z} = 4$, which is a contradiction.
	
	Thus we know that all vertices $v$ of $\Gamma$ have $\beta_v \in \{1,(t-1)/2,(t+1)/2 \}$.
	By Lemma~\ref{lem:equitable}, the valency partition of $\Gamma$ is equitable.
	Let $Q$ be the quotient matrix given as
	\[
		Q = \begin{pmatrix}
			0 & k_{12} & k_{13} \\
			k_{21} & k_{22} & k_{23} \\
			k_{31} & k_{32} & k_{33}
		\end{pmatrix}.
	\]  
	Note that since the number of common neighbours between any two vertices is nonnegative, no pair of vertices $v$ and $w$ can be adjacent if $\beta_v = \beta_w = 1$.
	Since the vector $(2, t-1, t+1)^\perp$ is an eigenvector of $Q$ for the eigenvalue $s$ we have 
	\begin{align*}
		(t-1)k_{12} + (t+1)k_{13} &= 2s; \\
		k_{12} + k_{13} &= t + 2.
	\end{align*}
	Hence $s = (t+2)(t-1)/2 + k_{13}$.
	Furthermore we have
	\begin{align*}
		2 k_{21} + (t-1)k_{22} + (t+1)k_{23} &= s(t-1); \\
		k_{21} + k_{22} + k_{23} &= (t^2 + 1)/2.
	\end{align*}
	Therefore, $s \leqslant \frac{t^2+1}{2}\frac{t+1}{t-1} < (t+1)^2/2 + 2$.
	Combining this bound with our expression for $s$ gives $k_{13} < (t+7)/2$ and hence $k_{12} > (t-3)/2 \geqslant 4$.
	
	Let $n_1$, $n_2$, and $n_3$ denote the number of vertices $v$ with $\beta_v = 1$, $\beta_v = (t-1)/2$, and  $\beta_v = (t+1)/2$ respectively.
	Since the valency partition of $\Gamma$ is equitable and $k_{12}$ and $k_{13}$ are positive, we have $n_1(t+2) = n_1(k_{12} + k_{13}) = n_2k_{21} + n_3k_{31} \geqslant n_2 + n_3$.
	Moreover the number of vertices of $\Gamma$ is at least $d_y = (t+1)^2/2 + t$.
	It follows that $n_1 > 1$ and hence there must exist a vertex $z \in C_2 \cup C_3$ with $\beta_z = 1$.
	
	First assume that $z \in C_3$.
	Since $k_{12} > \nu_{x,z}$, there must exist a vertex $w \in C_2 \backslash N_\Gamma(x)$ such that $w \sim z$ and $\beta_w = (t-1)/2$.
	Let $u$ be a neighbour of $x$ with $\beta_u = (t+1)/2$.
	We must have that $u$ is adjacent to $w$.
	Indeed, if $u$ were not adjacent to $w$ then $\nu_{u,w} = d_w - 1$ which implies $u$ has at least $\nu_{w,x} - 1 = t-2$ common neighbours with $x$.
	But $u$ has only $\nu_{u,x} = 2$ common neighbours with $x$. 
	Now, since $\nu_{u,z} = 0$, the vertex $z$ cannot be adjacent to any vertex $v \in N_\Gamma(x)$ with $\beta_v = (t+1)/2$.
	Therefore $d_z \geqslant \nu_{x,z} + \nu_{y,z} = 2+ t+1 = t+3$, which is impossible.
	
	Otherwise assume that there are no vertices $v \in C_3$ with $\beta=1$.
	Therefore all vertices $v \in C_3$ have $\beta_v = (t+1)/2$.
	Take $z \in C_2$ with $\beta_z = 1$.
	Then $z$ has at least $d_z - \nu_{x,z} - \nu_{y,z} = t - 2$ neighbours in $C_3$.
	Hence $k_{13} \geqslant t$ which is a contradiction.
	 
	 By Claim~1 we can assume that for all $w \in N_\Gamma(x)$, we have $\beta_w = (t-1)/2$.
	 Note that, since $y$ is adjacent to $x$ and $\beta_y=(t-1)/2$, we have $\nu_{x,y} = 0$.
	 Furthermore the number of triangles containing $x$ is $0$, which means that $s = (t^2+t-2)/2$.
	 Let $w$ be a vertex in $C_3$.
	 By Lemma~\ref{lem:alphabound} we see that $\beta_w \leqslant (t+1)/2$, and since $w$ is adjacent neither $x$ nor $y$ we must have $\beta_w \leqslant (t-1)/2$.
	 Moreover, since $d_w \geqslant \nu_{w,x} + \nu_{w,y} - \nu_{x,y}$, we have that $\beta_w = 1$ or $(t-1)/2$.
	 
	 Since $\Delta(\Gamma) > s = (t^2+t-2)/2$, there must be a vertex $z \in C_2$ that has $\beta_z = (t+1)/2$.
	 Therefore $z$ has $2t$ neighbours in $C_3$.
 	Let $a$ and $b$ denote the number of vertices $v \in C_3$ having $\beta_v = 1$ and $(t-1)/2$ respectively.
 	Set $S = C_2 \cap N_\Gamma(x)$.
 	By counting the common neighbours with $x$ we see that the number of edges between $S$ and $C_3$ is equal to $2a + (t-1)b$.
 	For each vertex $w \in S$, the number of neighbours in $C_3$ is given by the formula $d_w - \nu_{w,x} - \nu_{w,y} + \nu_{x,y}$.
 	Therefore $2a+(t-1)b = t^2 +t$ and so $b \leqslant t+2 < 2t$.
 	Consequently $z$ must have a neighbour $v \in C_3$ with $\beta_v = 1$.
 	But this violates Claim~1.
\end{proof}

In the last part of this section we establish bounds for $t$.

\begin{lemma}\label{lem:a-4}
	Let $\Gamma$ be a graph in $\mathcal H^\prime(s,t)$ and let $x$ be a vertex with valency $\Delta(\Gamma)$.
	Suppose $\Gamma$ has a vertex $y$ with $\alpha_y \leqslant \alpha_x - 4$
	Then $t \leqslant 29$.
\end{lemma}
\begin{proof}
	If $\alpha_y \geqslant 4$ then it follows from Corollary~\ref{cor:tboundal} and Corollary~\ref{cor:tbound} that $t \leqslant 22$.
	Hence, we assume that $\alpha_y < 4$.
	We are free to assume that $t \geqslant 11$, in which case, by Theorem~\ref{thm:kmin2}, we have $\alpha_y \geqslant \sqrt{3}$.
	If $\omega(\Gamma) \geqslant 3$ then we can apply Corollary~\ref{cor:psige3} together with Corollary~\ref{cor:tbound} to obtain $t \leqslant 21$.
	Hence we can assume that $\omega(\Gamma) = 1$ or $2$.
	Therefore, $\alpha_y$ is equal to $2$, $3$, or $2\sqrt{2}$.
	
	Here we split into two cases for $\omega(\Gamma) = 1$ or $2$.
	Since the arguments are similar, we will only consider the case when $\omega(\Gamma) = 1$.
	
	Let $S$ be the set of vertices $v$ of $\Gamma$ with $\alpha_v = 2$ or $3$ and let $T$ denote $V(\Gamma)\backslash S$.
	Every vertex $w$ in $T$ has $\alpha_w \geqslant 4$, therefore, by the argument above, we can assume that $\alpha_w \geqslant \alpha_x - 3$, otherwise we are done.
	Now, if there is a vertex $z \in T$ that is not adjacent to a vertex $y \in S$ then by Lemma~\ref{lem:alphabound} we have
	$t \geqslant (\alpha_z-\alpha_y)\alpha_y \geqslant (\alpha_x-3-\alpha_y)\alpha_y$ which gives
	\[
		\alpha_x \leqslant 
		\begin{cases}
			(t+18)/3, & \text{ if $\alpha_y=3$}; \\
			(t+10)/2, & \text{ if $\alpha_y=2$}. \\
		\end{cases}
	\]
	Applying Corollary~\ref{cor:tbound} gives $t \leqslant 19$ and $t \leqslant 29$ respectively.
	
	Otherwise, suppose $y$ is adjacent to every vertex in $T$.
	This means that $|T| = \alpha_y^2 + t$.
	Since we have assumed that $t \geqslant 11$, the subgraph induced on $S$ contains no edges.
	Therefore, by interlacing $|S| \leqslant m$.
	Hence, using Equation~\eqref{eqn:multeq} and Lemma~\ref{lem:rhobound}, we have $|S| \leqslant \frac{n-1+s}{t+1} \leqslant \frac{n+d_x}{t+1} \leqslant \frac{2n}{t+1}$.
	Furthermore, by Corollary~\ref{cor:nbound}, we have $|S| \leqslant 2t + 15$.
	Also by Corollary~\ref{cor:nbound} we have $\frac{1}{2}(t+\frac{1}{2})^2 < n = |S|+|T| \leqslant \alpha_y^2 + 3t + 15$.
	It follows that $t \leqslant 11$.
\end{proof}

Let us examine the proof of the previous result.
In order to bound $t$ we first found bounds for $\alpha_x$ then applied Corollary~\ref{cor:tbound}.
Since we assumed that $t \geqslant 11$ we see that in each part of the proof the bound $\alpha_x \leqslant (t+10)/2$ holds. 

\begin{lemma}\label{lem:5vals}
	Let $\Gamma$ be a graph in $\mathcal H^\prime(s,t)$ having at least $5$ distinct valencies.
	Then $t \leqslant 29$.
\end{lemma}
\begin{proof}
	Let $x$ be a vertex with $d_x = \Delta(\Gamma)$.
	Since $\Gamma$ has at least $5$ distinct valencies, there must be a vertex $y$ having $\alpha_y \leqslant \alpha_x - 4$.
	Then the result follows from Lemma~\ref{lem:a-4}.
\end{proof}

The next result is the main result of this section.

\begin{theorem}\label{thm:4vals}
	Let $\Gamma$ be a graph in $\mathcal H^\prime(s,t)$.
	Then $t \leqslant 29$.
\end{theorem}
\begin{proof}
	By Corollary~\ref{cor:psige3} and Corollary~\ref{cor:tbound}, we can assume that $\omega(\Gamma) = 1$ or $2$.
	Hence we split into two cases for $\omega(\Gamma) = 1$ or $2$.
	Since the two cases are similar, we will leave the case $\omega(\Gamma) = 2$ to the reader.
	
	Let $x$ be a vertex with $d_x = \Delta(\Gamma)$.
	If there is a vertex $y$ with $\alpha_y \leqslant \alpha_x - 4$ then the result follows from Lemma~\ref{lem:a-4}.
	We can therefore assume for each vertex $v$ we have $\alpha_v \in \{\alpha_x - 3,\alpha_x - 2, \alpha_x - 1, \alpha_x \}$.
	Moreover, by Theorem~\ref{thm:kmin2}, we can assume $\alpha_x \geqslant 6$.
	Define the sets $S_i := \{ v \in V(\Gamma) | \alpha_v = \alpha_x - i \}$ and set  $n_i = |S_i|$.
	By Corollary~\ref{cor:tboundal}  and Corollary~\ref{cor:tbound}, we can assume that each vertex in $S_0$ is adjacent to every vertex in $S_2 \cup S_3$.
	Observe that, since $\Gamma$ has at least three distinct valencies, the set $S_3 \cup S_2$ is nonempty.
	Hence we take $y \in S_3$ if $n_3 \ne 0$ otherwise we take $y \in S_2$.
	Let $x_i$ (resp. $y_i$) denote the number of neighbours of $x$ (resp. $y$) in $S_i$.
	Note that $y_0 = n_0$, $x_2 = n_2$, and $x_3 = n_3$.
	 
	Using the Perron-Frobenius eigenvector we have the following equations.
	\begin{align*}
		n_3(\alpha_x-3) + n_2(\alpha_x-2) + x_1(\alpha_x-1) + x_0\alpha_x &= s \alpha_x \\
		n_3+n_2+x_1+x_0 = \alpha_x^2 + t.
	\end{align*}
	This reduces to 
	\begin{equation}
		\label{eqn:x}
		3n_3+2n_2+x_1 = \alpha_x(\alpha_x^2+t-s).
	\end{equation}
	We also have
	\begin{align*}
		y_3(\alpha_x-3) + y_2(\alpha_x-2) + y_1(\alpha_x-1) + n_0\alpha_x &= s \alpha_y \\
		y_3+y_2+y_1+n_0 = \alpha_y^2 + t,
	\end{align*}
	which gives 
	\begin{equation}
		\label{eqn:y}
		y_2+2y_1+3n_0 = \alpha_y(s-\alpha_y^2-t).
	\end{equation}
	Now the left hand side of the sum of Equations~\eqref{eqn:x} and \eqref{eqn:y} is at most $3n$.
	Since the same is true for the sum of the right hand sides of Equations~\eqref{eqn:x} and \eqref{eqn:y}, we obtain the bound 
	\[
		n \geqslant 
		\begin{cases}
			2\alpha_x^2 - 9\alpha_x+9, & \text{ if $\alpha_y = \alpha_x - 3$;} \\
			2\alpha_x^2 - 5\alpha_x-1, & \text{ if $\alpha_y = \alpha_x - 2$.}
		\end{cases}
	\]
	
	We continue considering the two cases separately.
	First suppose $\alpha_y = \alpha_x - 2$.
	Combining the above with Corollary~\ref{cor:nbound}, we have $2\alpha_x^2 - 5\alpha_x-1 \leqslant n < t^2 + 8t + 18 + 18/(t-1)$.
	Moreover, since $\alpha_x \leqslant t + 1$, we obtain the bound $\alpha_x^2 < (t^2 + 13t+ 24)/2 + 9/(t-1)$.
	We can assume that $t \geqslant 11$.
	And, since $\alpha_x^2$ is an integer, we have $\alpha_x^2 < (t^2 + 13t+ 26)/2 < (t+7)^2/2$.
	Therefore $\alpha_x < (t+7)/\sqrt{2}$.
	
	Applying Lemma~\ref{lem:boundVerts} with $\alpha_x < (t+7)/\sqrt{2}$ gives
	\[
		n \leqslant \frac{t^2+(24-\sqrt{2})t + 125-9\sqrt{2}}{2} + \frac{168-16\sqrt{2}}{2(t-1)}.
	\]
	Therefore, again using our lower bound for $n$, we obtain
	\[
		2\alpha_x^2 - 5\alpha_x-1 \leqslant \frac{t^2+(24-\sqrt{2})t + 129-9\sqrt{2}}{2}.
	\]
	Hence we use the bound $\alpha_x \leqslant (t+13+2\sqrt{2})/2$.
	This time, applying Lemma~\ref{lem:boundVerts} and following the same method, we can obtain the bound $\alpha_x \leqslant (t+10)/2$.
	The lemma then follows from Corollary~\ref{cor:tbound}.
	
	The case when $\alpha_y = \alpha_x - 3$ is similar but is easier since one can begin by using the bound from Lemma~\ref{lem:alphabound}.
	That is, $3(\alpha_x-3) = (\alpha_x - \alpha_y) \leqslant 2(t-1)$, and hence $\alpha_x \leqslant (2t+7)/3$.
	Then apply Lemma~\ref{lem:boundVerts} and so on.
\end{proof}

	Just as we saw in the proof of Lemma~\ref{lem:a-4}, we also have that if $t \geqslant 11$ then everywhere in the proof of Theorem~\ref{thm:4vals} the bound $\alpha_x \leqslant (t+10)/2$ holds.
	We record this observation as a corollary. 
	\begin{corollary}\label{cor:nalphabound}
		Let $\Gamma$ be a graph in $\mathcal H^\prime(s,t)$ with $t \geqslant 11$ and let $x$ be a vertex with $d_x = \Delta(\Gamma)$.
		Then $\alpha_x \leqslant (t+10)/2$.
	\end{corollary}


\section{Description of our computation} 
\label{sec:description_of_our_computation}

In this section we describe our computations.
In the previous section, we reduced the proof of Theorem~\ref{thm:main} to a finite search.
Indeed, for an $n$-vertex graph $\Gamma \in \mathcal H^\prime(s,t)$, Corollary~\ref{cor:nbound} provides us with a quadratic upper bound for the number of vertices $n$ in terms of the smallest eigenvalue $-t$.
Together with Theorem~\ref{thm:4vals}, it follows that there are only finitely many graphs in $\cup_{t,s \in \mathbb N}\mathcal H^\prime(s,t)$.

Our method can be described loosely in the following way.
First enumerate tuples $(n,t,s,m)$ consisting of values for the number of vertices, smallest eigenvalue, largest eigenvalue, and eigenvalue multiplicity that satisfy certain spectral conditions.
Then for each such tuple we enumerate all possible valencies.
Finally we check for feasible valency multiplicities.

Now we will describe our computation in detail.

\subsection{Constraints for the spectrum} 
\label{sub:spectral_bounds}

Fix $t \geqslant 3$ and let $\Gamma$ be an $n$-vertex graph in $\mathcal H^\prime(s,t)$.
We can bound $n$ in terms of $t$ using Corollary~\ref{cor:nbound}.
Unfortunately, this bound becomes less practically useful as $t$ gets large.
Fortunately, using Corollary~\ref{cor:nalphabound} together with Corollary~\ref{cor:tbound}, we can obtain a better bound for $n$ when $t \geqslant 11$.
That is, we obtain the following bounds.

\begin{corollary}\label{cor:compnbounds}
	Let $\Gamma$ be an $n$-vertex graph in $\mathcal H^\prime(s,t)$.
	Then
	\[
		n \leqslant 
		\begin{cases}
			t^2 + 8t + 18 + 18/(t-1), & \text{ if $3 \leqslant t \leqslant 10$;} \\
			t^2/4 + 17t/2 + 48 + 116/(t-1), & \text{ if $11 \leqslant t \leqslant 29$.} \\
		\end{cases}
	\]
\end{corollary}

As for the other two spectral parameters, $s$ and $m$, they are linked via the equation
$s=1-n+m(t+1)$; this is Equation~\eqref{eqn:multeq}.
We also use the bounds from Lemma~\ref{lem:rhobound}, Theorem~\ref{thm:bell}, and Corollary~\ref{cor:maxDeg}.
Furthermore, it is well-known that the sum of the squares of the eigenvalues is equal to two times the number of edges and the sum of the cubes of the eigenvalues is equal to six times the number of triangles.
The last result we use is that $s \leqslant n - 6$ \cite[Lemma 2.13]{CGGK}.

Putting this altogether we give the following definition.

\begin{definition}\label{def:spec}
	Fix $t \in \mathbb \{3,\dots,29\}$.
	We call a triple $(n,s,m)$ a \textbf{spectral parameter array} if the following conditions hold.
	\begin{enumerate}[(a)]
		\item If $3 \leqslant t \leqslant 10$ then $n \leqslant t^2 + 8t + 18 + 18/(t-1)$ otherwise if $11 \leqslant t \leqslant 29$ then $n \leqslant t^2/4 + 17t/2 + 48 + 116/(t-1)$;
		\item $(n-m)(n-m+1) \geqslant 2n$ and $(m+1)(m+2) \geqslant 2n$;
		\item $n < s^2 + 1$;
		\item $t < s < \min\{(t+1)^2 + t,n-6\}$;
		\item $s^2 + n-1-m + mt^2$ is even and less than $ns$;
		\item $s^3 + n-1-m - mt^3$ is divisible by $6$.
	\end{enumerate}
\end{definition}

Let  $\mathcal S(t)$ denote the set of spectral parameter arrays for $t$.
The number of spectral parameter arrays for each $t$ satisfying $3 \leqslant t \leqslant 29$ is given in Table~\ref{tab:specparams}.


\subsection{Constraints for the valencies} 
\label{sub:valency_constraints}

Now we consider the set of possible valencies for a putative graph $\Gamma$ with given spectral parameter array $(n,s,m) \in \mathcal S(t)$.
Let $v$ be a vertex of $\Gamma$.
Then clearly $\alpha_v \geqslant 1$ and, by Corollary~\ref{cor:maxDeg}, we have $\alpha_v \leqslant t+1$.
Using Theorem~\ref{thm:kmin} and Theorem~\ref{thm:kmin2}, if $t \geqslant 7$ or $t \geqslant 11$ then $\alpha_v\geqslant 2$ or $\alpha \geqslant 3$ respectively.
Furthermore, by Corollary~\ref{cor:intparams}, we have $\alpha_v \in \sqrt{\omega(\Gamma)}\mathbb N$.
By Lemma~\ref{lem:rhobound} the largest and smallest valencies must be respectively greater than and less than $s$.
We also use the bounds in Lemma~\ref{lem:alphabound}.
In particular, for any two vertices $v$ and $w$, if $(\alpha_v-\alpha_w)\alpha_w > t$ then $v$ must be adjacent to $w$.
Hence, for such vertices, we check Lemma~\ref{lem:boundVerts}, that is, $d_v + d_w - \nu_{v,w} \geqslant n - 2m$.
Moreover, $\Gamma$ is connected so there must exist a pair of vertices satisfying that bound.
Since $\Gamma$ is not a cone, for each vertex there must exist another vertex that is not adjacent to it. 
Let $x$ be a vertex with $d_x = \Delta(\Gamma)$.
Then, in particular, by Theorem~\ref{thm:disconn}, there must exist some vertex $v$ such that $(\alpha_x-\alpha_v)\alpha_v \leqslant t$.
Finally, using Equation~\eqref{eqn:triangles}, we derive the expression $(s-t+1)\alpha_v^2-(t-1)t$ for the number of closed walks of length three from $v$.
This number must be a nonnegative even integer for all vertices $v$.

We will use the following bounds for the valency multiplicities.

\begin{lemma}\label{lem:nmaxbound}
	Let $\Gamma$ be a graph in $\mathcal H^\prime(s,t)$ with distinct valencies $k_1 < k_2 < \dots < k_r$ and corresponding (nonzero) valency multiplicities $n_1,\dots,n_r$.
	Set $T = s^2 + n-1-m + mt^2$.
	Then for all $2 \leqslant i \leqslant r$ we have
	\begin{align}
		n_i &\leqslant \left ( T -nk_1 -\sum_{j=i+1}^{r}n_j (k_j - k_{1})- \sum_{j=1}^{i-1}(k_j-k_{1})\right )/(k_i-k_{1}); \label{eqn:upp} \\
		n_i &\geqslant \left ( T- nk_{i-1} -\sum_{j=i+1}^{r}n_j (k_j - k_{i-1}) - \sum_{j=1}^{i-1}(k_j-k_{i-1})\right )/(k_i-k_{i-1}). \label{eqn:low}
	\end{align}
\end{lemma}
\begin{proof}
	From the trace of the square of the adjacency matrix we have the equality $T = \sum_{j=1}^r n_j k_j$.
	Hence we can write $n_i k_i = T - \sum_{j=i+1}^{r}n_j k_j - \sum_{j=1}^{i-1}n_j k_j $.
	The upper bound \eqref{eqn:upp} follows since
	\begin{align*}
		\sum_{j=1}^{i-1}n_j k_j &\geqslant \sum_{j=1}^{i-1}n_j k_{1} + \sum_{j=1}^{i-1}(k_{j}-k_{1}) \\
								&=(n-n_i)k_1 - \sum_{j=i+1}^{r}n_j k_1 + \sum_{j=1}^{i-1}(k_{j}-k_{1}).
	\end{align*}
	The lower bound \eqref{eqn:low} follows similarly.
\end{proof}

In particular, by Lemma~\ref{lem:nmaxbound}, given valencies $k_1 < k_2 < \dots < k_r$ and sum of the squares of the eigenvalues, $T$, we have that $n_r(k_r-k_1) \leqslant T - n k_1 - \sum_{i=1}^{r-1}(k_i-k_1)$.
Since the valency multiplicities are positive this bound gives us another condition on the valencies since $n_r \geqslant 1$.

Putting these valency conditions together, we give our next definition.
For convenience we will use the function $c_t(a,b) := \sqrt{(a-t)(b-t)}$.

\begin{definition}\label{def:val}
	Fix $t \in \mathbb \{3,\dots,29\}$ and take some $(n,s,m) \in \mathcal S(t)$.
	An $r$-tuple $(k_1,\dots,k_r)$ with $k_1 < \dots < k_r$ and $r \geqslant 3$ is called a $(n,s,m)$-\textbf{feasible valency array} if following conditions are satisfied.
	\begin{enumerate}[(a)]
		\item For all $i \in \{1,..,r\}$ the squarefree part of $(k_i-t)$ is constant.
		\item $k_r \leqslant (t+1)^2 + t$ and 
		\[
			k_1 \geqslant \begin{cases}
				t+3 & \text{ if $t \geqslant 11$;} \\
				t+2 & \text{ if $t \geqslant 7$;} \\
				t + 1 & \text{ otherwise. }
			\end{cases}
		\]
		\item For each pair $i,j \in \{1,\dots,r\}$ with $i > j$ we have $c_t(k_i,k_j) - (k_j-t) \leqslant 2t -2$.
		\item If $c_t(k_i,k_j) - (k_j-t) > t$ then $k_i + k_j - c_t(k_i,k_j) + t - 1 \geqslant n - 2m$.
		\item For all $i \in \{1,..,r\}$, the expression $(s-t+1)(k_i-t)-(t-1)t$ is a nonnegative even integer.
		\item There exists $i\in \{1,\dots,r-1\}$ such that $c_t(k_r,k_i) - (k_i-t) \leqslant t$.
		\item There exist $i,j\in \{1,\dots,r\}$ such that $k_i + k_j - c_t(k_i,k_j) + t - 1 \geqslant n - 2m$.
		\item And $s^2 + n-1-m + mt^2 - n k_1 - \sum_{i=2}^r(k_i-k_1) \geqslant k_r-k_1$.
	\end{enumerate}
\end{definition}

Let $\mathcal K(t)$ denote the subset of $\mathcal S(t)$ consisting of those spectral parameter arrays $S$ such that there exists a $S$-feasible valency array.
The cardinality of $\mathcal K(t)$ for each $t$ satisfying $3 \leqslant t \leqslant 29$ is given in Table~\ref{tab:specparams}.
 

\subsection{Constraints for the valency multiplicities} 
\label{sub:valency_multiplicities}

The final check we need to do is with the valency multiplicities.
Suppose there exists an $n$-vertex graph $\Gamma$ with eigenvalues $s$, $1$, and $-t$ with multiplicities $1$, $n-1-m$, and $m$ respectively and valencies $k_1< \dots< k_r$ with multiplicities $n_1, \dots, n_r$ respectively.
Let $\alpha$ be an eigenvector for $s$ satisfying Equation~\eqref{eqn:alpha}.
Then by summing the entries of the equations $A\alpha = s \alpha$ and Equation~\eqref{eqn:alpha} we obtain the equations
\begin{align}
	\sum_{i=1}^r n_i k_i \alpha_i &= s \sum_{i=1}^r n_i \alpha_i; \label{eqn:mult1}\\
	\left (\sum_{i=1}^r n_i \alpha_i \right )^2 &= \sum_{i=1}^r n_i (k_i-1)(k_i+t). \label{eqn:mult2}
\end{align}

Using Equation~\eqref{eqn:alpha} and Equation~\eqref{eqn:triangles} we can also obtain the equations
\begin{align}
	\sum_{i=1}^r n_i k_i &= s^2 + n-1-m + mt^2; \label{eqn:mult3}\\
	\sum_{i=1}^r n_i \left ((s-t+1)(k_i-t)-(t-1)t \right ) &= s^3 + n-1-m - mt^3. \label{eqn:mult4}
\end{align}

Finally we check the independence number.
Suppose there exists $h \in \{1,\dots,r\}$ such that for all $i,j \in \{1,\dots,h\}$ we have $k_i + k_j - c_t(k_i,k_j) + t - 1 < n - 2m$.
Every pair of vertices with degrees in $\{k_1,\dots,k_h\}$ are not adjacent and hence form an independent set.
Therefore there exists and independent set of size $\sum_{i=1}^h n_i$.
By Lemma~\ref{lem:indep}, the size of this set is at most $m$.
If such an $h$ exists, set $\mathfrak h = h$, otherwise set $\mathfrak h = 0$.
By above we have that $\sum_{i=1}^\mathfrak{h} n_i \leqslant m$.

Putting these conditions together we give our final definition.

\begin{definition}\label{def:mult}
	Fix $t \in \mathbb \{3,\dots,29\}$, $S \in \mathcal K(t)$, and an $S$-feasible valency array $K=\{k_1,\dots,k_r\}$.
	An $r$-tuple $(n_1,\dots,n_r)$ is called a $(S,K)$-\textbf{feasible multiplicity array} if Equations \eqref{eqn:mult1}, \eqref{eqn:mult2}, \eqref{eqn:mult3}, and \eqref{eqn:mult4} are satisfied, Lemma~\ref{lem:nmaxbound} is satisfied, and $\sum_{i=1}^\mathfrak{h} n_i \leqslant m$.
\end{definition}

Let $\mathcal M(t)$ denote the subset of $\mathcal K(t)$ consisting of those spectral parameter arrays $S$ such that there exists a $(S,K)$-feasible multiplicity array for some $S$-feasible valency array $K$.
The cardinality of $\mathcal M(t)$ for each $t$ satisfying $3 \leqslant t \leqslant 29$ is given in Table~\ref{tab:specparams}.

In order to enumerate all possible $(S,K)$-feasible multiplicity arrays in a reasonable amount of time, one can use Lemma~\ref{lem:nmaxbound} to recursively compute the bounds for consecutive $n_i$.
Indeed, we see that in Lemma~\ref{lem:nmaxbound}, if $n_r,\dots,n_{r-l}$ are fixed for some $l\in \{0,\dots,r-2\}$ we can compute the bounds for $n_{r-l-1}$ and we can bound $n_r$ in terms of $t$ and elements of $S$ and $K$. 
We remark that we wrote our computer program in \texttt{MAGMA}~\cite{Magma} and our computation will run on a standard personal computer in a matter of hours.
The code is available online~\cite{WEB}.

\begin{table}[htbp]
	\setlength{\tabcolsep}{4pt}
	\begin{center}
	\begin{tabular}{c|c|c|c||c|c|c|c||c|c|c|c}
		$t$ & $|\mathcal S(t)|$ & $|\mathcal K(t)|$ & $|\mathcal M(t)|$ &$t$ & $|\mathcal S(t)|$ & $|\mathcal K(t)|$ & $|\mathcal M(t)|$ & $t$ & $|\mathcal S(t)|$ & $|\mathcal K(t)|$ & $|\mathcal M(t)|$ \\
		\hline
		3 & 128 & 58 & 0    &       12 & 497 & 287 & 0    &        	21 & 189 & 137 & 0    \\
		4 & 196 & 116 & 1   & 		13 & 455 & 237 & 0    &  		22 & 163 & 137 & 0    \\
		5 & 277 & 113 & 2   & 		14 & 409 & 245 & 0    &  		23 & 143 & 120 & 0    \\
		6 & 375 & 173 & 0   & 		15 & 377 & 214 & 0    &  		24 & 118 & 104 & 0    \\
		7 & 492 & 159 & 1   & 		16 & 340 & 220 & 0    &  		25 & 95 & 92 & 0      \\
		8 & 610 & 225 & 0   & 		17 & 311 & 184 & 0    &  		26 & 76 & 71 & 0      \\
		9 & 748 & 233 & 0   & 		18 & 273 & 190 & 0    &  		27 & 61 & 59 & 0      \\
		10 & 898 & 297 & 0  & 		19 & 248 & 162 & 0    &  		28 & 43 & 43 & 0      \\
		11 & 546 & 272 & 0  & 		20 & 220 & 172 & 0    &  		29 & 27 & 27 & 0      \\
	\end{tabular}
	\end{center}
	\caption{Cardinality of the sets $\mathcal S(t)$, $\mathcal K(t)$, and $\mathcal M(t)$ for each $t$.}
	\label{tab:specparams}
\end{table}

\subsection{Surviving parameters} 
\label{sub:surviving_parameters}

As one can see in Table~\ref{tab:specparams}, our computation returned four spectral parameter arrays that have feasible valency arrays and corresponding feasible multiplicity arrays.
We display these arrays in Table~\ref{tab:survive}.
\begin{table}[htbp]
	\begin{center}
	\begin{tabular}{c|c|c|c}
		$t$ & $(n,s,m)$ & $(k_1,\dots,k_r)$ & $(n_1,\dots,n_r)$ \\
		\hline
		$4$ & $(31,15,9)$ & $(5,8,13,20)$ & $(5,10,5,11)$ \\
		$5$ & $(36,19,9)$ & $(7,13,23)$ & $(6,12,18)$ \\
		$5$ & $(45,28,12)$ & $(6,9,21,30)$ & $(6,3,3,33)$ \\
		$7$ & $(45,20,8)$ & $(11,16,23,32)$ & $(6,27,6,6)$
	\end{tabular}
	\end{center}
	\caption{The surviving parameter arrays.}
	\label{tab:survive}
\end{table}

For each of the parameters in Table~\ref{tab:survive}, we will show why there cannot exist a corresponding graph.

\begin{itemize}
	\item There does not exist any graph having spectrum $\{15^1,1^{21},(-4)^9\}$, valencies $(5,8,13,20)$, and multiplicities $(5,10,5,11)$.
	Suppose to the contrary.
	Let $x$ be a vertex with degree $d_x = 5$.
	Using Equation~\eqref{eqn:triangles}, we see that there are $0$ closed walks of length three from $x$.
	Therefore each neighbour $y$ of $x$ must have $d_y \leqslant 13$.
	On the other hand, by Lemma~\ref{lem:alphabound}, we have that each neighbour $y$ of $x$ must have degree $d_y \geqslant 13$.
	Hence, each vertex of degree $5$ is adjacent to each of the $5$ vertices of degree $13$.
	Let $w$ be another vertex of degree $5$.
	Then the number of common neighbours of $w$ and $x$ is $5$, but this is a contradiction since $\nu_{w,x} = 1$.
	
	\item There does not exist any graph having spectrum $\{19^1,1^{26},(-5)^9\}$, valencies $(7,13,23)$, and multiplicities $(6,12,18)$.
	Suppose to the contrary.
	Then, by Lemma~\ref{lem:equitable}, the valency partition is equitable and hence each vertex in $V_1$ must be adjacent to
	$k_{11}$ vertices in $V_1$, $k_{12}$ vertices in $V_2$ and $k_{13}$ vertices in $V_3$ such that $k_{11} + k_{12}+k_{13} = k_1$.
	Moreover (see \cite[Section 4.3]{Dam3ev}), we also have the equations $\sum_{j=1}^3 k_{1j}\sqrt{(k_j-t)} = s\sqrt{(k_1-t)}$ and
	\begin{equation*}
		\sum_{j=1}^3 k_{1j}k_j = (1-t)k_1 +t + \sqrt{(k_1-t)}\sum_{j=1}^3 n_j \sqrt{(k_j-t)}.
	\end{equation*}
	Let $x$ and $y$ be vertices with degrees $d_x = d_y = k_1 = 7$.
	The vertex $x$ cannot be adjacent to $y$. 
	Indeed, if $x$ were adjacent to $y$ then $\nu_{x,y} = -2$ but this number should be nonnegative.
	Hence $k_{11} = 0$ and there is no solution to the system of equations above.
	
	\item There does not exist any graph having spectrum $\{28^1,1^{32},(-5)^{12}\}$, valencies $(6,9,21,30)$, and multiplicities $(6,3,3,33)$.
	Suppose to the contrary.
	Since the number of common neighbours of any pair of vertices is nonnegative, the only possible neighbours of vertices in $V_1$ are vertices in $V_3$ or $V_4$.
	Let $x$ be a vertex with degree $d_x = 6$.
	Using Equation~\eqref{eqn:triangles}, we know that there are four closed walks of length $3$ from $x$.
	Moreover, $x$ has precisely one common neighbour with each adjacent vertex in $V_4$.
	Therefore, $x$ can have at most four neighbours in $V_4$.
	Hence each vertex has at least two neighbours in $V_3$.
	Since each vertex in $V_1$ has precisely one common neighbour there must be at least as many vertices in $V_3$ as there are in $V_1$.
	But $|V_1| = 6 > 3 = |V_3|$.
	
	\item There does not exist any graph having spectrum $\{20^1,1^{36},(-7)^{8}\}$, valencies $(11,16,23,32)$, and multiplicities $(6,27,6,6)$.
	Suppose to the contrary.
	This graph corresponds to the case of equality in the bound in Theorem~\ref{thm:bell}.
	But the case of equality in the bound in Theorem~\ref{thm:bell} is known \cite{br03} to correspond to an unique graph on $36$ vertices.
\end{itemize}



\bibliographystyle{myplain}
\bibliography{sbib}

\end{document}